\documentclass[12pt,reqno]{amsart}
\usepackage{amsmath,amsthm,amssymb,amsfonts}
\usepackage[hidelinks]{hyperref}
\usepackage{geometry}
\usepackage{etoolbox}
\usepackage{orcidlink}

\numberwithin{equation}{section}

\theoremstyle{plain}
\newtheorem{theorem}{Theorem}[section]
\newtheorem{lemma}[theorem]{Lemma}
\newtheorem{corollary}[theorem]{Corollary}

\theoremstyle{remark}
\newtheorem{remark}[theorem]{Remark}

\newcommand{\dif}{\mathrm{d}}
\newcommand{\majorarc}{\mathfrak{M}}
\newcommand{\minorarc}{\mathfrak{m}}

\begin{document}
\title{Prime pairs along rays of prime indices}

\author{Shenghao Hua~\orcidlink{0000-0002-7210-2650}}
\address{Shanghai Institute for Mathematics and Interdisciplinary Sciences (SIMIS), Shanghai 200433, China}
\address{Research Institute of Intelligent Complex Systems, Fudan University, Shanghai, 200433, China}
\email{huashenghao@vip.qq.com}

\author{Sizhe Xie}
\address{Morningside Center of Mathematics, Academy of Mathematics and Systems Science, Chinese Academy of Sciences, Beijing 100190, China}
\email{szxie@amss.ac.cn}

\begin{abstract}
Let $p_j$ be the prime with index $j$.  We prove that the ratios $m/n$ for which $p_m+p_n$ is a square are dense in $\mathbb R_{>0}$.
The same is true when $|p_m-p_n|$ is a square.  In every nonempty open interval, almost every index can be used as a numerator and as a denominator.
We also give a quantitative lower bound for the number of possible partners.
\end{abstract}



\maketitle

\section{Introduction}

Questions about sums and differences of primes are basic in number theory.
The Goldbach conjecture says that every even integer $N\geq4$ is a sum of two primes.
The twin prime conjecture says that there are infinitely many prime pairs of the form $(p,p+2)$.  Both problems remain open.
Chen proved that every sufficiently large even integer is a sum of a prime and a number with at most two prime factors~\cite{Chen1966,Chen1973}.
Zhang proved that bounded gaps between primes occur infinitely often~\cite{Zhang2014}.
Later work of Maynard and the Polymath project improved the size of the gap~\cite{Maynard2015,Polymath2014,PolymathSieve2014}.  If
\begin{equation*}
 H_1=\liminf_{n\to\infty}(p_{n+1}-p_n),
\end{equation*}
then the Polymath project proved that $H_1\leq246$~\cite[Theorem~4(i)]{PolymathSieve2014}.
Very recently, Stadlmann improved this
to $H_1\leq240$ \cite{Stadlmann2026}.

A conjecture of Zhi-Wei Sun~\cite[Conjecture 4.4]{Sun2017}, of a different flavor, involves the indices of primes rather than the primes themselves.
It asserts that for every positive rational number $r$, there exist positive integers $m,n$ such that $r=m/n$ and $p_m+p_n=k^2$ for some integer $k$, where $p_i$ denotes the $i$-th prime.
The same assertion is made with $p_m+p_n$ replaced by $|p_m-p_n|$.

The requirement that the ratio be fixed exactly makes these conjectures difficult.
Suppose that $r=a/b$ in lowest terms.
Every pair satisfying $m/n=r$ must be of the form $(m,n)=(ta,tb)$ for some positive integer $t$.
Thus one must find a value of $t$ for which $p_{ta}+p_{tb}$ or $|p_{ta}-p_{tb}|$ is a square.  Although each fixed value of $t$ can be checked, no finite search can disprove the conjecture, since a
solution may occur farther along the sequence.  Moreover, there is no simple formula for $p_n$ in terms of $n$, and present analytic methods do not control
these square conditions along such a thin sequence.
In this paper, we allow the ratio to vary inside any given open interval.
We begin with a qualitative version of our main result.

\begin{theorem}\label{thm:density}
Let $p_j$ be the prime with index $j$.
Put
\begin{align*}
 \mathcal R_\Sigma
 &:=\left\{\frac mn:m,n\geq1,\ p_m+p_n\text{ is a square}\right\},\\
 \mathcal R_\Delta
 &:=\left\{\frac mn:m,n\geq1,\ |p_m-p_n|
        \text{ is a square}\right\}.
\end{align*}
Then both $\mathcal R_\Sigma$ and $\mathcal R_\Delta$ are dense in $\mathbb R_{>0}$.
\end{theorem}

This theorem gives only the density statement.  Our full result says more.
In every open interval, almost every fixed index has many partners.
We now introduce the notation needed for the stronger statement.
For a nonempty open interval $I\subset\mathbb R_{>0}$, define
\begin{align*}
 r_{\Sigma,D}(n,I)
 &:=\#\left\{m\geq1:\frac mn\in I,\ p_m+p_n\text{ is a square}\right\},\\
 r_{\Sigma,N}(m,I)
 &:=\#\left\{n\geq1:\frac mn\in I,\ p_m+p_n\text{ is a square}\right\},\\
 r_{\Delta,D}(n,I)
 &:=\#\left\{m\geq1:\frac mn\in I,\
 |p_m-p_n|\text{ is a square}\right\},\\
 r_{\Delta,N}(m,I)
 &:=\#\left\{n\geq1:\frac mn\in I,\
 |p_m-p_n|\text{ is a square}\right\}.
\end{align*}

\begin{theorem}[Quantitative version]\label{thm:main}
For every nonempty open interval $I\subset\mathbb R_{>0}$,
there is a constant $c_I>0$ such that, for every fixed $L>0$,
as $N\to\infty$,
\begin{align}
 \#\left\{n\leq N:r_{\Sigma,D}(n,I)\geq
 c_I\frac{\sqrt n}{(\log(2n))^{3/2}}\right\}
 &=N+O_{I,L}\left(\frac{N}{(\log N)^L}\right),
 \label{sum-many-D}\\
 \#\left\{m\leq N:r_{\Sigma,N}(m,I)\geq
 c_I\frac{\sqrt m}{(\log(2m))^{3/2}}\right\}
 &=N+O_{I,L}\left(\frac{N}{(\log N)^L}\right).
 \label{sum-many-N}
\end{align}
The same constant $c_I$ may be chosen so that
\begin{align}
 \#\left\{n\leq N:r_{\Delta,D}(n,I)\geq
 c_I\frac{\sqrt n}{(\log(2n))^{3/2}}\right\}
 &=N+O_{I,L}\left(\frac{N}{(\log N)^L}\right),
 \label{difference-many-D}\\
 \#\left\{m\leq N:r_{\Delta,N}(m,I)\geq
 c_I\frac{\sqrt m}{(\log(2m))^{3/2}}\right\}
 &=N+O_{I,L}\left(\frac{N}{(\log N)^L}\right).
 \label{difference-many-N}
\end{align}
Here $I$ and $L$ are fixed as $N\to\infty$.
\end{theorem}

Theorem~\ref{thm:density}  follows immediately from
Theorem~\ref{thm:main}.

\begin{remark}
The following heuristic suggests such an average result.
When $m,n\asymp N$, the sum $p_m+p_n$ has size about $N\log N$.
The same is true of $|p_m-p_n|$ when $m/n$ stays away from $1$.
The chance that an integer of this size is a square is about $1/\sqrt{N\log N}$.
Since there are about $N$ possible choices for the other index, this suggests about $\sqrt{N/\log N}$ representations.
\end{remark}

The analytic input for our argument consists of mean square
estimates for quadratic prime values, extending the work of
Baier and Zhao~\cite{BaierZhao2007}.
In the broader setting of the Birch--Goldbach problem, Liu and
the second named author~\cite{LiuXieForms} study nonsingular systems of forms of
differing degrees in sufficiently many prime variables.
Related developments include their saving-transfer treatment
of enlarged major arcs~\cite{LiuXieSavingTransfer} and their
bounds for multiple Gauss sums~\cite{LiuXieGauss}.
Here we consider the mixed equations $p+q=h^2$ and $|p-q|=h^2$,
with $p,q$ prime and $h$ an integer.
Our aim is to obtain many partners for almost every prime
endpoint, and we proceed by averaging over the shift in the
associated quadratic prime value problems.

\section{Estimates for quadratic progressions}\label{sec:input}

We first change the condition on the indices into a condition on the
sizes of the primes.  Let $u$ and $v$ lie in fixed compact subintervals
of $\mathbb R_{>0}$.  The prime number theorem gives the following
estimate uniformly
\begin{equation}\label{index-transfer}
 \frac{\pi(uX)}{\pi(vX)}
 =\frac{u}{v}\frac{\log(vX)}{\log(uX)}(1+o(1))
 =\frac{u}{v}+o(1).
\end{equation}
Consequently, for any fixed $R>0$ and any open interval $I$
containing $R$, a sufficiently thin fixed rectangle about
$(RX,X)$ has the following property: every prime pair
$(p_m,p_n)$ in the rectangle satisfies $m/n\in I$ for all
sufficiently large $X$.

The square conditions naturally lead to three quadratic prime value problems. After fixing one of the two primes and writing the square as $h^2$,
we express the other prime in one of the forms $h^2+k$, $h^2-k$,
and $k-h^2$, depending on the sign and on which prime is fixed.
For \(h^2+k\) in the range \(y\leq x^2\), the required mean square estimate is due to Baier and Zhao~\cite{BaierZhao2007}; their later work gives stronger positive-shift results in other ranges~\cite{BaierZhao2009}. Our application requires the extended range \(y\leq K_0x^2\) for each fixed \(K_0>1\), as well as analogous estimates for \(h^2-k\) and \(k-h^2\). These estimates, together with the corresponding local and archimedean factors, are established in Section~\ref{sec:negative}.

For $u\in\mathbb R$, write $u_+:=\max\{u,0\}$.  We set
$\Lambda(t)=0$ when $t\leq0$.  We also put
\begin{equation}\label{H0-definition}
 H_0(x,k):=\min\{x,\sqrt{k}\}.
\end{equation}
For every positive integer $k$, we define
\begin{align}
 \mathfrak S_+(k)
  &:=\prod_{p>2}\left(1-\frac{\left(\frac{-k}{p}\right)}{p-1}\right),
  \label{positive-series}\\
 \mathfrak S_-(k)
  &:=\prod_{p>2}\left(1-\frac{\left(\frac{k}{p}\right)}{p-1}\right).
  \label{negative-series}
\end{align}
Here $\left(\frac{a}{p}\right)$ denotes the Legendre symbol
for an integer $a$ and an odd prime $p$.

\begin{theorem}[Extension of the theorem of Baier and Zhao to a fixed multiple]
\label{thm:positive-main}
Let $K_0\geq1$ be fixed.  Given $A,B>0$, uniformly for
\[
 x^2(\log x)^{-A}\leq y\leq K_0x^2,
\]
one has
\begin{equation*}\label{positive-mean-square}
 \sum_{\substack{k\leq y\\ \mu^2(k)=1}}
 \left|
   \sum_{h\leq x}\Lambda(h^2+k)-\mathfrak S_+(k)x
 \right|^2
 \ll_{A,B,K_0}\frac{yx^2}{(\log x)^B}.
\end{equation*}
\end{theorem}

\begin{theorem}[Estimate for the negative shift]\label{thm:negative-main}
Given $A,B>0$, uniformly for
\[
 x^2(\log x)^{-A}\leq y\leq x^2,
\]
one has
\begin{equation*}\label{negative-mean-square}
 \sum_{\substack{k\leq y\\ \mu^2(k)=1\\k>1}}
 \left|
   \sum_{h\leq x}\Lambda(h^2-k)
   -\mathfrak S_-(k)(x-\sqrt{k})_+
 \right|^2
 \ll_{A,B}\frac{yx^2}{(\log x)^B}.
\end{equation*}
\end{theorem}

\begin{theorem}[Estimate for the reverse shift]\label{thm:reverse-main}
Let $K_0\geq1$ be fixed.
Given $A,B>0$, uniformly for
\[
 x^2(\log x)^{-A}\leq y\leq K_0x^2,
\]
one has
\begin{equation*}\label{reverse-mean-square}
 \sum_{\substack{k\leq y\\ \mu^2(k)=1\\k>1}}
 \left|
   \sum_{h\leq x}\Lambda(k-h^2)
   -\mathfrak S_-(k)H_0(x,k)
 \right|^2
 \ll_{A,B,K_0}\frac{yx^2}{(\log x)^B}.
\end{equation*}
\end{theorem}

The case $K_0=1$ of Theorem~\ref{thm:positive-main} is the original
theorem of Baier and Zhao~\cite[Theorem, p.~964]{BaierZhao2007}.
Section~\ref{sec:negative} proves the extension for fixed $K_0$ and
also proves Theorems~\ref{thm:negative-main} and
\ref{thm:reverse-main}.  Stronger results for the original positive
shift are given by Baier and Zhao~\cite{BaierZhao2009}.

\begin{remark}\label{rem:fixed-constant}
Once the interval $I$ is chosen, the number $K_0$ is fixed.  We neither
claim nor need uniformity as $K_0\to\infty$.  Each compact subinterval
of $(0,1)$ or $(1,\infty)$ requires only one fixed value of $K_0$.
\end{remark}

\begin{lemma}\label{lem:series-lower}
For either sign and every squarefree $k>1$,
\begin{equation}\label{series-lower}
 \mathfrak S_\pm(k)\gg\frac1{\log(2k)},
\end{equation}
with an absolute implied constant.
\end{lemma}

\begin{proof}
Let $a=-k$ for $\mathfrak S_+(k)$ and $a=k$ for
$\mathfrak S_-(k)$.  Let $D_a$ be the fundamental discriminant of
$\mathbb Q(\sqrt a)$.  Since $k>1$ is squarefree, the primitive real
character $\chi_{D_a}$ is nonprincipal.  Its conductor is $|D_a|$,
and $|D_a|\leq4k$.  It also satisfies
\[
 \chi_{D_a}(p)=\left(\frac ap\right)
 \qquad(p>2).
\]
Writing $\chi=\chi_{D_a}$, we have
\begin{align*}
 \mathfrak S_\pm(k)L(1,\chi)
 &=(1-\chi(2)/2)^{-1}
 \prod_{p>2}
 \frac{1-\chi(p)/(p-1)}{1-\chi(p)/p}.
\end{align*}
For $p>2$, this factor is positive and is
$1+O(p^{-2})$ uniformly for $\chi(p)\in\{-1,0,1\}$.  Thus the Euler
product is bounded above and below by positive absolute constants.
The same holds for the factor at $2$.  Hence
\[
 \mathfrak S_\pm(k)L(1,\chi_{D_a})\asymp1.
\]
The elementary bound $L(1,\chi_{D_a})\ll\log(2|D_a|)$ yields
\eqref{series-lower}.
\end{proof}

Chebyshev's inequality and Lemma~\ref{lem:series-lower} give the
following pointwise consequences of the mean value estimates.

\begin{corollary}\label{cor:pointwise}
Fix $A,C,D>0$.  In each of Theorems~\ref{thm:positive-main},
\ref{thm:negative-main}, and~\ref{thm:reverse-main}, take the saving
parameter $B$ larger than $C+2D$.
In the range of Theorem~\ref{thm:positive-main}, the following formula
holds for all but $O(y(\log x)^{-C})$ squarefree integers $k\leq y$.
\begin{equation}\label{positive-absolute-corollary}
 \sum_{h\leq x}\Lambda(h^2+k)
 =\mathfrak S_+(k)x
 +O\left(\frac{x}{(\log x)^D}\right).
\end{equation}

In the range of Theorem~\ref{thm:negative-main}, the following formula
holds for all but $O(y(\log x)^{-C})$ squarefree integers $k\leq y$
with $k>1$.
\begin{equation}\label{negative-absolute-corollary}
 \sum_{h\leq x}\Lambda(h^2-k)
 =\mathfrak S_-(k)(x-\sqrt{k})_+
+O\left(\frac{x}{(\log x)^D}\right).
\end{equation}
For any $E\geq 0$, this is a relative asymptotic whenever
\[
x-\sqrt{k}\geq x(\log x)^{-E}
\quad\hbox{and}\quad D>E+1.
\]

In the range of Theorem~\ref{thm:reverse-main}, the following formula
holds for all but $O(y(\log x)^{-C})$ squarefree integers $k\leq y$
with $k>1$.
\begin{equation}\label{reverse-absolute-corollary}
 \sum_{h\leq x}\Lambda(k-h^2)
 =\mathfrak S_-(k)H_0(x,k)
+O\left(\frac{x}{(\log x)^D}\right).
\end{equation}
For any $E\geq 0$, this is a relative asymptotic whenever
\[
H_0(x,k)\geq x(\log x)^{-E}
\quad\hbox{and}\quad D>E+1.
\]
\end{corollary}

\begin{proof}
Chebyshev's inequality, applied with threshold $x(\log x)^{-D}$ in the relevant mean square estimate, gives
$O(y(\log x)^{-C})$ exceptions, provided that $B>C+2D$. The relative statements follow from
Lemma~\ref{lem:series-lower} by comparing the main terms with the
absolute errors.
\end{proof}

\section{Circle method estimates for the three quadratic polynomials}
\label{sec:negative}

We treat the three forms simultaneously.  For the tails of the
singular series, the nonprincipal term in the prime exponential sum,
and the contribution from the minor arcs, we follow the treatment
of Baier and Zhao~\cite{BaierZhao2007}.  We modify the local factors and the
archimedean factors for the two signed forms.  The quadratic phase prevents a direct application of  Gallagher's linear lemma to the quadratic
nonprincipal term.  To overcome this difficulty, we establish a weighted $T_1E_2$
estimate and use the regrouping in
\eqref{recombined-major-decomposition}.  This constitutes the principal new ingredient in
our circle method argument.

Put $e(z)=e^{2\pi iz}$.  We need two standard estimates for short intervals.  We state them in
the exact forms used below.  These forms also keep track of the end
intervals in Gallagher's lemma.

\begin{lemma}[Short-interval estimates]\label{lem:short-interval-input}
Let $2<\Delta<N/2$, $N<N'\leq2N$, and let  $(a_n)$ be any finite
sequence.  Then
\begin{equation}\label{Gallagher-exact}
 \int_{|\beta|\leq\Delta^{-1}}
 \left|\sum_{N<n\leq N'}a_ne(\beta n)\right|^2\,\dif\beta
 \ll \Delta^{-2}
 \int_{N-\Delta/2}^{N'}
 \left|
 \sum_{\max\{t,N\}<n\leq\min\{t+\Delta/2,N'\}}a_n
 \right|^2\,\dif t.
\end{equation}
The implied constant is absolute.

For $q\geq1$ and $\Delta>0$, put
\begin{equation}\label{J-definition}
 J_N(q,\Delta)=
 \sideset{}{^{\#}}\sum_{\chi\bmod q}\int_N^{2N}
 \left|\sum_{t<n\leq t+q\Delta}
       \chi(n)\Lambda(n)\right|^2\,\dif t,
\end{equation}
where $\#$ means that we replace the inner sum, in the principal-character term, by
$\sum_{t<n\leq t+q\Delta}(\Lambda(n)-1)$.
Given $\varepsilon,L,b>0$, uniformly for
\begin{equation*}\label{Wolke-Mikawa-range}
 q\leq(\log N)^b,
 \qquad N^{1/5+\varepsilon}\leq\Delta\leq N^{1-\varepsilon},
\end{equation*}
one has
\begin{equation}\label{Wolke-Mikawa-input}
 J_N(q,\Delta)\ll_{\varepsilon,L,b}
 (q\Delta)^2N(\log N)^{-L}.
\end{equation}
\end{lemma}

\begin{proof}
Inequality~\eqref{Gallagher-exact} is Gallagher's lemma
\cite{Gallagher1970}, in the form stated in
\cite[Lemma~1]{BaierZhao2007}.  The short-interval estimate
\eqref{Wolke-Mikawa-input} is due to Mikawa~\cite{Mikawa1991},
building on methods of Wolke~\cite{Wolke1989}; we use the formulation
in \cite[Lemma~2]{BaierZhao2007}.  In this formulation, the estimate is
uniform in $q$ throughout the stated logarithmic range, and the sum runs
over all Dirichlet characters modulo $q$, including imprimitive ones.
\end{proof}

\subsection{Local factors}

For an odd prime $p$, let
\[
 \rho_p^-(k)=\#\{r\pmod p:r^2-k\equiv0\pmod p\}.
\]
The root count is
\begin{equation*}\label{root-count-minus}
 \rho_p^-(k)=1+\left(\frac{k}{p}\right).
\end{equation*}
Thus the normalized local factor is
\begin{equation*}\label{local-factor-minus}
 \frac{1-\rho_p^-(k)/p}{1-1/p}
 =1-\frac{(k/p)}{p-1},
\end{equation*}
which proves~\eqref{negative-series}.  The same congruence appears for
$k-h^2$.  In fact, a prime divisor of $k-h^2$ satisfies
$h^2\equiv k\pmod p$.  For $h^2+k$, the congruence is
$h^2\equiv-k\pmod p$.  The symbol is then $(-k/p)$.  This gives
\eqref{positive-series}.  Modulo $2$, every residue $k$ has exactly one
solution of $r^2\equiv k\pmod2$.  Thus the normalized factor at $p=2$
is one in all three cases.

For an odd positive integer $q$, we write
$\left(\frac{a}{q}\right)$ for the Kronecker symbol,
with $\left(\frac{a}{1}\right)=1$. The singular series admit the Dirichlet series expansions
\begin{align}
 \mathfrak S_-(k)
 &=\lim_{V\to\infty}
   \sum_{\substack{q\leq V\\2\nmid q}}
   \frac{\mu(q)}{\varphi(q)}
   \left(\frac{k}{q}\right),
 \label{negative-series-expansion}\\
 \mathfrak S_+(k)
 &=\lim_{V\to\infty}
   \sum_{\substack{q\leq V\\2\nmid q}}
   \frac{\mu(q)}{\varphi(q)}
   \left(\frac{-k}{q}\right).
 \label{positive-series-expansion}
\end{align}
Here the limits are taken through the natural partial sums $q\leq V$,
following \cite[Section~4]{BaierZhao2007}.  The corresponding series
may fail to converge absolutely, so the order of summation is essential.
Their coefficients are related by the fixed
quadratic twist
\begin{equation*}\label{fixed-twist}
 \left(\frac{k}{q}\right)
 =
 \left(\frac{-1}{q}\right)
 \left(\frac{-k}{q}\right)
 \qquad (2\nmid q).
\end{equation*}

If $k=t^2$ with $t\geq1$, then $(k/p)=1$ for every prime $p\nmid t$,
and hence
\[
 \mathfrak S_-(t^2)
 =
 \prod_{\substack{p>2\\p\nmid t}}
 \left(1-\frac{1}{p-1}\right)
 =0.
\]
This vanishing reflects the factorizations
\[
 h^2-t^2=(h-t)(h+t),
 \qquad
 t^2-h^2=(t-h)(t+h).
\]
Since $k$ is restricted to be squarefree, the only square value of $k$
is $1$.  We therefore exclude $k=1$ from the negative- and
reverse-shift estimates.

\subsection{A unified circle method setup}

In this subsection, $K_0$ is fixed.  The implied constants may depend
on the fixed parameters in the statement.  They do not depend on $x$,
$y$, or $k$.

We place the three circle method arguments in a common framework.
For $\nu\in\{+,-,0\}$, define
\begin{equation*}\label{signed-data}
 (s_+,t_+)=(-1,-1),\qquad
 (s_-,t_-)=(-1,1),\qquad
 (s_0,t_0)=(1,-1),
\end{equation*}
and set
\begin{equation}\label{three-cutoffs}
 z_+=x^2+y,\qquad z_-=x^2,\qquad z_0=y.
\end{equation}
We further define
\[
 \varkappa_+=-1,\qquad
 \varkappa_-=\varkappa_0=1,
\]
and, recalling \eqref{H0-definition},
\[
 H_+(x,k)=x,\qquad
 H_-(x,k)=(x-\sqrt{k})_+,\qquad
 H_0(x,k)=\min\{x,\sqrt{k}\}.
\]
In this notation, the equation
\[
 m+s_\nu h^2+t_\nu k=0
\]
specializes, for $\nu=+,-,0$, respectively, to
\[
 m=h^2+k,\qquad
 m=h^2-k,\qquad
 m=k-h^2.
\]
For $s\in\{-1,1\}$, set
\[
 S_{1,z}(\alpha)=\sum_{m\leq z}\Lambda(m)e(\alpha m),
 \qquad
 S_{2,s}(\alpha)=\sum_{h\leq x}e(s\alpha h^2).
\]
By orthogonality,
\begin{equation*}\label{unified-fourier-identity}
\begin{aligned}
 \mathcal P_\nu(k)
 &:=\sum_{h\leq x}\Lambda(-s_\nu h^2-t_\nu k)\\
 &=\int_0^1 S_{1,z_\nu}(\alpha)S_{2,s_\nu}(\alpha)
   e(t_\nu k\alpha)\,\dif\alpha.
\end{aligned}
\end{equation*}
In all three applications,
\begin{equation}\label{z-range}
 \frac{x^2}{(\log x)^A}\ll z_\nu\ll_{K_0}x^2,
 \qquad \log z_\nu\asymp_{K_0}\log x.
\end{equation}

Fix
\begin{equation}\label{circle-parameters}
 Q_1=(\log x)^c,\qquad Q=x^{1-\delta},
 \qquad 0<\delta<\frac1{10},
\end{equation}
where we choose $c$ at the end.  On the interval
$\mathcal I=[1/Q,1+1/Q]$, let
\[
 \majorarc=\bigcup_{q\leq Q_1}
 \bigcup_{\substack{1\leq a\leq q\\(a,q)=1}}
 \left[\frac aq-\frac1{qQ},\frac aq+\frac1{qQ}\right],
 \qquad
 \minorarc=\mathcal I\setminus\majorarc.
\]
The arcs are disjoint for sufficiently large $x$.

We first decompose $S_{1,z}$ according to coprimality with the major arc
denominator.  Let
\[
 \alpha=\frac{a}{q}+\beta\in\majorarc.
\]
Then
\begin{equation}\label{S1-coprime-splitting}
 S_{1,z}(\alpha)
 =
 \sum_{\substack{m\leq z\\(m,q)=1}}
 \Lambda(m)e\left(\frac{am}{q}\right)e(\beta m)
 +
 \sum_{\substack{m\leq z\\(m,q)>1}}
 \Lambda(m)e\left(\frac{am}{q}\right)e(\beta m).
\end{equation}
Since $\Lambda$ is supported on prime powers, the second sum is
\begin{equation}\label{S1-noncoprime-part}
 \sum_{p\mid q}
 \sum_{\substack{j\geq1\\p^j\leq z}}
 \log p\,
 e\left(\frac{ap^j}{q}\right)e(\beta p^j).
\end{equation}

We now treat the first sum in \eqref{S1-coprime-splitting}.  For a
Dirichlet character $\chi\bmod r$, define its Gauss sum by
\[
 \tau_r(\chi)
 :=
 \sum_{b\bmod r}\chi(b)e(b/r).
\]
Since $(a,q)=1$, character orthogonality gives, for $(m,q)=1$,
\begin{equation}\label{character-expansion}
e\left(\frac{am}{q}\right)
 =
 \frac{1}{\varphi(q)}
 \sum_{\chi\bmod q}
 \tau_q(\chi)\overline{\chi}(am)
 =
 \frac{1}{\varphi(q)}
 \sum_{\chi\bmod q}
 \tau_q(\chi)\overline{\chi}(a)\overline{\chi}(m).
\end{equation}
Since $\tau_q(\chi_0)=\mu(q)$, set
\[
 T_{1,z}(\alpha)
 :=
 \frac{\mu(q)}{\varphi(q)}
 \sum_{m\leq z}e(\beta m).
\]
For each character $\chi\bmod q$, define
\begin{equation}\label{def Achi}
 A_\chi(m):=
 \begin{cases}
  \Lambda(m)-1, & \chi=\chi_0,\\
  \overline{\chi}(m)\Lambda(m), & \chi\neq\chi_0,
 \end{cases}
\end{equation}
and set
\[
 E_{1,z}(\alpha)
 :=
 \frac{1}{\varphi(q)}
 \sum_{\chi\bmod q}
 \tau_q(\chi)\overline{\chi}(a)
 \sum_{m\leq z}A_\chi(m)e(\beta m).
\]
The character identity above and the definitions of $T_{1,z}$ and
$E_{1,z}$ give
\begin{align*}
 &\sum_{\substack{m\leq z\\(m,q)=1}}
 \Lambda(m)e\left(\frac{am}{q}\right)e(\beta m)\\
 &\qquad=
 T_{1,z}(\alpha)+E_{1,z}(\alpha)
 -
 \frac{\mu(q)}{\varphi(q)}
 \sum_{\substack{m\leq z\\(m,q)>1}}
 \Lambda(m)e(\beta m).
\end{align*}
Combining this identity with \eqref{S1-noncoprime-part}, we obtain
\begin{equation}\label{S1-decomposition}
 S_{1,z}(\alpha)
 =
 T_{1,z}(\alpha)+E_{1,z}(\alpha)+R_{q,z}(\alpha),
\end{equation}
where
\begin{align*}
 R_{q,z}(\alpha)
 &=
 \sum_{\substack{m\leq z\\(m,q)>1}}
 \Lambda(m)
 \left(
  e\left(\frac{am}{q}\right)
  -\frac{\mu(q)}{\varphi(q)}
 \right)e(\beta m)
 \notag\\
 &=
 \sum_{p\mid q}
 \sum_{\substack{j\geq1\\p^j\leq z}}
 \log p
 \left(
  e\left(\frac{ap^j}{q}\right)
  -\frac{\mu(q)}{\varphi(q)}
 \right)e(\beta p^j).
 \label{S1-remainder}
\end{align*}
Consequently, it follows from $q\leq Q_1$ and \eqref{z-range} that
\begin{equation*}\label{S1-remainder-bound}
 |R_{q,z}(\alpha)|
 \ll \omega(q)\log(2z)
 \ll \bigl(\log(2z)\bigr)^2.
\end{equation*}

For the quadratic sum $S_{2,s}(\alpha)$, let $d\mid q$ and write
\[
 q^*=\frac qd,\qquad
 d_{q^*}=\frac d{(d,q^*)},\qquad q^*_{d}=\frac{q^*}{(q^*,d)}.
\]
If $(h,q)=d$, then $h=dn$, $(n,q^*)=1$, $\left(sa d_{q^*},q^*_{d}\right)=1$, and
\[
 e\left(\frac{sah^2}{q}\right)
 =
 e\left(\frac{sa d_{q^*}n^2}{q^*_{d}}\right).
\]
Similarly to \eqref{character-expansion}, we obtain
\begin{equation}\label{quadratic-character-expansion}
 e\left(\frac{sa d_{q^*}n^2}{q^*_{d}}\right)
 =
 \frac{1}{\varphi(q^*_{d})}
 \sum_{\chi\bmod q^*_{d}}
 \tau_{q^*_{d}}(\chi)
 \overline{\chi}(sa d_{q^*})
 \overline{\chi}^{\,2}(n).
\end{equation}
Splitting $S_{2,s}(\alpha)$ according to $d=(h,q)$ gives
\[
 S_{2,s}(\alpha)
 =
 \sum_{d\mid q}
 \sum_{\substack{h\leq x\\(h,q)=d}}
 e\left(\frac{sah^2}{q}\right)e(s\beta h^2).
\]
We insert \eqref{quadratic-character-expansion}, with $n=h/d$, and
separate the characters according as $\chi^2=\chi_0$ or
$\chi^2\neq\chi_0$.  This gives
\begin{equation}\label{S2-decomposition}
 S_{2,s}(\alpha)
 =
 T_{2,s}(\alpha)+E_{2,s}(\alpha),
\end{equation}
where
\begin{align}
 T_{2,s}(\alpha)
 &=
 \sum_{d\mid q}\frac{1}{\varphi(q^*_{d})}
 \sum_{\substack{\chi\bmod q^*_{d}\\\chi^2=\chi_0}}
 \tau_{q^*_{d}}(\chi)
 \overline{\chi}(sa d_{q^*})
 \sum_{\substack{h\leq x\\(h,q)=d}}
 e(s\beta h^2),
 \label{T2-definition}\\
 E_{2,s}(\alpha)
 &=
 \sum_{d\mid q}\frac{1}{\varphi(q^*_{d})}
 \sum_{\substack{\chi\bmod q^*_{d}\\\chi^2\neq\chi_0}}
 \tau_{q^*_{d}}(\chi)
 \overline{\chi}(sa d_{q^*})
 \sum_{\substack{h\leq x\\(h,q)=d}}
 \overline{\chi}^{\,2}(h/d)e(s\beta h^2).
 \label{E2-definition}
\end{align}
The coefficient appearing in $T_{2,s}$ has a convenient
complete-sum interpretation.  Define
\[
 G_s(a,q,d)
 :=
 \sum_{\substack{\ell\bmod q^*_{d}\\(\ell,q^*_{d})=1}}
 e\left(\frac{sa d_{q^*}\ell^2}{q^*_{d}}\right).
\]
Applying \eqref{character-expansion} once more and summing
over the reduced residue classes $\ell\bmod q^*_{d}$, we find
\begin{align*}
 G_s(a,q,d)
 &=
 \frac{1}{\varphi(q^*_{d})}
 \sum_{\chi\bmod q^*_{d}}
 \tau_{q^*_{d}}(\chi)
 \overline{\chi}(sa d_{q^*})
 \sum_{\substack{\ell\bmod q^*_{d}\\(\ell,q^*_{d})=1}}
 \overline{\chi}^{\,2}(\ell)\\
 &=
 \sum_{\substack{\chi\bmod q^*_{d}\\\chi^2=\chi_0}}
 \tau_{q^*_{d}}(\chi)
 \overline{\chi}(sa d_{q^*}),
\end{align*}
since
\[
 \sum_{\substack{\ell\bmod q^*_{d}\\(\ell,q^*_{d})=1}}
 \overline{\chi}^{\,2}(\ell)
 =
 \begin{cases}
  \varphi(q^*_{d}),&\chi^2=\chi_0,\\
  0,&\chi^2\neq\chi_0.
 \end{cases}
\]
Consequently,
\begin{equation*}\label{T2-complete-sum}
 T_{2,s}(\alpha)
 =
 \sum_{d\mid q}
 \frac{G_s(a,q,d)}{\varphi(q^*_{d})}
 \sum_{\substack{h\leq x\\(h,q)=d}}
 e(s\beta h^2).
\end{equation*}
We have $|G_s(a,q,d)|\leq\varphi(q^*_{d})$.  The conditions $(h,q)=d$
split the sum over $h$.  Therefore
\begin{equation*}\label{T-bounds}
 |T_{1,z}(\alpha)|\leq z,
 \qquad |T_{2,s}(\alpha)|\leq x.
\end{equation*}
Also, the decompositions for $s=+1$ are the complex conjugates of the
decompositions for $s=-1$.

\subsection{Singular series and principal major arcs}

\begin{lemma}[Singular series tails]\label{lem:tail-bound}
Let $\varkappa\in\{-1,+1\}$ and
\[
 \mathcal T_\varkappa(k,Q_1)=
 \lim_{V\to\infty}
 \sum_{\substack{Q_1<q\leq V\\2\nmid q}}
 \frac{\mu(q)}{\varphi(q)}
 \left(\frac{\varkappa k}{q}\right).
\]
Recall that $Q_1=(\log x)^c$ in \eqref{circle-parameters}.
For every $C>0$, we can take $c=c(C,A,K_0)$ sufficiently large.  Then
\begin{equation}\label{tail-second-moment}
 \sum_{\substack{k\leq y\\\mu^2(k)=1\\k>1}}
 |\mathcal T_\varkappa(k,Q_1)|^2
 \ll_{C,A,K_0}\frac{y}{(\log x)^C}.
\end{equation}
\end{lemma}

\begin{proof}
Choose $0<\eta\leq1/5$.  Put
\[
 U=\sqrt y,\qquad W=\exp(y^{\eta/2}),
\]
and split the sum over $q$ into three ranges:
\[
 Q_1<q\leq U,\qquad U<q\leq W,\qquad q>W.
\]
For the first range, after dropping the squarefree restriction and expanding the square, we estimate the diagonal and off-diagonal contributions as in \cite[(5.2)]{BaierZhao2007}.  They give
\begin{equation*}\label{tail-short-range}
 \sum_{\substack{k\leq y\\\mu^2(k)=1}}
 \left|\sum_{\substack{Q_1<q\leq U\\2\nmid q}}
 \frac{\mu(q)}{\varphi(q)}
 \left(\frac{\varkappa k}{q}\right)\right|^2
 \ll \frac y{Q_1}+U(\log(2U))^5.
\end{equation*}
The diagonal is bounded by
$y\sum_{q>Q_1}\mu^2(q)/\varphi(q)^2$.  The character sum bound of
P\'olya and Vinogradov controls the off-diagonal terms.  These estimates apply to both choices of $\varkappa$, since changing
the sign of $\varkappa$ merely multiplies the coefficient indexed
by each odd $q$ by $(\frac{-1}{q})$, leaving the diagonal terms
and the absolute bounds for the off-diagonal terms unchanged.

For the middle range, let
\begin{equation*}
 v=\left\lceil\log_2(W/U)\right\rceil.
\end{equation*}
We split the sum over $q$ into the intervals
$2^{r-1}U<q\leq2^rU$.  We use the real character large sieve in
\cite{HeathBrown1995}.  We also use the classical large sieve
in \cite[(5.3) to (5.6)]{BaierZhao2007}.  These two large sieves allow
arbitrary coefficients.  Thus the fixed factor $(-1/q)$ does not change the bounds.
We obtain
\begin{equation}\label{tail-middle-range}
 \sum_{\substack{k\leq y\\\mu^2(k)=1}}
 \left|\sum_{\substack{U<q\leq W\\2\nmid q}}
 \frac{\mu(q)}{\varphi(q)}
 \left(\frac{\varkappa k}{q}\right)\right|^2
\ll_\eta
 v\left(y^{3\eta}+U^{\eta-1}y^{1+\eta}
       +v+y^{1-\eta}\right)
 \ll_\eta y^{1-\eta/2}.
\end{equation}
We have $U=\sqrt y$ and $v\ll_\eta y^{\eta/2}$.  After we include the
outer factor $v$, the four terms in \eqref{tail-middle-range} are at
most
\[
 y^{7\eta/2},\qquad y^{1/2+2\eta},\qquad
 y^\eta,\qquad y^{1-\eta/2},
\]
in the same order.  Each term is $O(y^{1-\eta/2})$ when
$\eta\leq1/5$.

We next use a uniform estimate for the long tail.  Let $D>0$ be fixed.
Suppose that $\chi$ is a nonprincipal primitive real character with
conductor $r\leq(\log W)^D$.  Let $\widetilde\chi$ be induced modulo
$M$, where $r\mid M$ and $M/r\leq2$.  Then
\begin{equation}\label{twisted-mobius-tail}
 \lim_{V\to\infty}\sum_{W<q\leq V}
 \frac{\mu(q)\widetilde\chi(q)}{\varphi(q)}
 \ll_D \exp\bigl(-c_D\sqrt{\log W}\bigr).
\end{equation}
We give the details.  This is the usual estimate due to Siegel and Walfisz
for the twisted M\"obius function.  The extra weight $q/\varphi(q)$ is
harmless.  Put
\[
 B_{\widetilde\chi}(T)=
 \sum_{n\leq T}\mu(n)\widetilde\chi(n)\frac n{\varphi(n)}.
\]
For $\Re s>1$, the Dirichlet series has the factorization
\begin{align*}
 \sum_{n=1}^{\infty}
 \frac{\mu(n)\widetilde\chi(n)n/\varphi(n)}{n^s}
 &=\frac{G_{\chi,M}(s)}{L(s,\chi)},          \label{mobius-factorization}
\end{align*}
where
\[
 G_{\chi,M}(s)=
 \prod_{p\nmid M}
 \frac{1-\chi(p)\dfrac p{p-1}p^{-s}}
      {1-\chi(p)p^{-s}}
 \prod_{\substack{p\mid M\\p\nmid r}}
 (1-\chi(p)p^{-s})^{-1}.
\]
The first product converges absolutely and uniformly on
$\Re s\geq1/2+\varepsilon$, while the second product has at most one
factor.  The standard zero-free region shows that $L(s,\chi)$ has no zeros in
\[
 \sigma\geq1-\frac{c}{\log(r(|t|+2))},
\]
apart from a possible real exceptional zero. If $\beta$ is such a
zero, Siegel's theorem with exponent $1/(2D)$ and the bound
$r\leq(\log W)^D\leq(\log T)^D$ give
\[
 1-\beta\gg_D(\log T)^{-1/2}.
\]
Thus $L(s,\chi)$ has no zeros in the rectangle needed for moving the
Perron contour to
\[
 \sigma=1-c_D(\log T)^{-1/2},
 \qquad |t|\leq\exp(\sqrt{\log T}).
\]
The constant here is ineffective.  The standard bounds for
$L(s,\chi)^{-1}$ on this contour give the following uniform estimate
for $T\geq W$:
\begin{equation*}\label{weighted-mobius-partial-sum}
 B_{\widetilde\chi}(T)
 \ll_D T\exp\bigl(-c_D\sqrt{\log T}\bigr).
\end{equation*}
The constant is again ineffective.  Partial summation gives
\begin{align*}
 \sum_{W<n\leq V}
 \frac{\mu(n)\widetilde\chi(n)}{\varphi(n)}
 &=\frac{B_{\widetilde\chi}(V)}V
   -\frac{B_{\widetilde\chi}(W)}W
   +\int_W^V\frac{B_{\widetilde\chi}(t)}{t^2}\,\dif t\\
 &\ll_D\exp\bigl(-c_D'\sqrt{\log W}\bigr).
\end{align*}
This proves \eqref{twisted-mobius-tail}.  It also proves convergence in
the natural order.  This is the argument behind
\cite[(5.7)]{BaierZhao2007}.

Let $k>1$ be squarefree.  Let $D_\varkappa(k)$ be the fundamental
discriminant of $\mathbb Q(\sqrt{\varkappa k})$.
For odd $q$,
\begin{equation*}
 \left(\frac{\varkappa k}{q}\right)
 =\chi_{D_\varkappa(k)}(q),
 \qquad
 |D_\varkappa(k)|\leq4k.
\end{equation*}
This primitive character is nonprincipal.  Let $\chi_0^{(2)}$ be the
principal character modulo $2$.  We restrict $q$ to odd numbers by
replacing $\chi_{D_\varkappa(k)}$ with
$\chi_{D_\varkappa(k)}\chi_0^{(2)}$.  The new character is
nonprincipal modulo
\begin{equation*}
 \operatorname{lcm}(|D_\varkappa(k)|,2)\leq8k
\end{equation*}
and it is induced by $\chi_{D_\varkappa(k)}$.  Therefore its conductor does
not change, and its working modulus increases by at most a factor of
two.
Since
\begin{equation*}
 k\leq y=(\log W)^{2/\eta},
\end{equation*}
we have the following bound for sufficiently large $W$.
\[
 |D_\varkappa(k)|\leq4(\log W)^{2/\eta}
 \leq(\log W)^{3/\eta}.
\]
We now apply \eqref{twisted-mobius-tail} with $D=3/\eta$.  The estimate
is uniform and gives
\begin{equation*}\label{tail-long-range}
 \sum_{\substack{q>W\\2\nmid q}}
 \frac{\mu(q)}{\varphi(q)}
 \left(\frac{\varkappa k}{q}\right)
 \ll_\eta \exp\left(-c_\eta\sqrt{\log W}\right).
\end{equation*}
It follows that
\begin{align*}
 \sum_{\substack{k\leq y\\\mu^2(k)=1\\k>1}}
 |\mathcal T_\varkappa(k,Q_1)|^2
 &\ll_\eta
 \frac y{Q_1}
 +\sqrt y\,(\log(2y))^5
 +y^{1-\eta/2}
 +y\exp\left(-2c_\eta\sqrt{\log W}\right).
\end{align*}
We have $Q_1=(\log x)^c$ and $y\geq x^2(\log x)^{-A}$.  We also have
$\log y\asymp_{A,K_0}\log x$.  Hence the last three terms have a power
saving.  We take $c=c(C,A,K_0)$ sufficiently large.  This proves
\eqref{tail-second-moment}.
\end{proof}

We next evaluate the contribution of the principal terms on the
major arcs, using Lemma~\ref{lem:tail-bound} to control the error
in completing the singular series.

\begin{lemma}[Principal major arcs]\label{lem:principal-arcs}
For every $C>0$, provided that $c$ in $Q_1=(\log x)^c$
is sufficiently large in terms of $C$, $A$, and $K_0$, one has
\begin{equation}\label{principal-L2}
 \sum_{\substack{k\leq y\\\mu^2(k)=1\\k>1}}
 \left|\int_{\majorarc}T_{1,z_\nu}(\alpha)T_{2,s_\nu}(\alpha)
 e(t_\nu k\alpha)\,\dif\alpha
 -\mathfrak S_\nu(k)H_\nu(x,k)\right|^2
 \ll_{C,A,K_0}
 \frac{yx^2}{(\log x)^C}+yxQ(\log x)^{C_1},
\end{equation}
where
\[
 \mathfrak S_\nu(k)=
 \begin{cases}
  \mathfrak S_+(k),&\nu=+,\\
  \mathfrak S_-(k),&\nu=-,0,
 \end{cases}
\]
and where $C_1$ depends only on the fixed circle method parameters.
\end{lemma}

\begin{proof}
Fix $q$, $a$, and $d$.  The $\beta$ integral in the principal product
is
\[
 \int_{|\beta|\leq1/(qQ)}
 \left(\sum_{m\leq z_\nu}e(\beta m)\right)
 \left(\sum_{\substack{h\leq x\\(h,q)=d}}e(s_\nu\beta h^2)\right)
 e(t_\nu k\beta)\,\dif\beta.
\]
We extend this integral to a full period.  Orthogonality then gives
$m+s_\nu h^2+t_\nu k=0$.  The upper condition $m\leq z_\nu$ follows
from \eqref{three-cutoffs}.  The condition $m\geq1$ gives the intervals
\[
 \mathcal I_+(k)=[1,x],\qquad
 \mathcal I_-(k)=(\sqrt k,x],\qquad
 \mathcal I_0(k)=[1,x]\cap[1,\sqrt{k}).
\]
M\"obius inversion gives the following estimate uniformly for
$d\mid q$.
\begin{equation}\label{unified-coprime-count}
 \#\{h\in\mathcal I_\nu(k):(h,q)=d\}
 =\frac{\varphi(q/d)}qH_\nu(x,k)+O(\tau(q/d)).
\end{equation}

The moving intervals appear only after orthogonality on the full
period.  For the extension error, the quadratic sum is still the full
sum over $h\leq x$.  The factor $e(t_\nu k\beta)$ has absolute value
one.  A change of the sign $s_\nu$ only conjugates the quadratic sum.
Thus the estimate is uniform in $\nu$ and $k$.  Parseval's identity gives
\begin{equation}\label{quadratic-parseval}
 \int_0^1\left|
 \sum_{\substack{h\leq x\\(h,q)=d}}e(s\beta h^2)
 \right|^2\,\dif\beta
 =\#\{h\leq x:(h,q)=d\}\leq\frac xd+1.
\end{equation}
Outside the short arc, we have
$|\sum_{m\leq z}e(\beta m)|\ll\|\beta\|^{-1}$.  Cauchy's inequality
and \eqref{quadratic-parseval} give
\begin{equation}\label{beta-extension}
 \int_{1/(qQ)\leq|\beta|\leq1/2}
 \frac1{\|\beta\|}
 \left|\sum_{\substack{h\leq x\\(h,q)=d}}e(s_\nu\beta h^2)\right|
 \,\dif\beta
\ll(qQ)^{1/2}\left(\frac xd+1\right)^{1/2}.
\end{equation}
We have
\[
 |G_s(a,q,d)|\leq\varphi(q^*_{d}).
\]
Thus the total contribution of the error in
\eqref{unified-coprime-count} is
\begin{align*}
 \mathcal E_{\mathrm{count}}
 &\leq \sum_{q\leq Q_1}\frac{|\mu(q)|}{\varphi(q)}
       \sum_{a\bmod q}^{*}\sum_{d\mid q}
       \frac{|G_{s_\nu}(a,q,d)|}{\varphi(q^*_{d})}
       \tau(q/d)\notag\\
 &\ll \sum_{q\leq Q_1}\sum_{d\mid q}\tau(q/d)
 \ll Q_1(\log(2Q_1))^2
 \ll(\log x)^{C_1}.                            \label{counting-error-sum}
\end{align*}
The sum over $a$ cancels the factor $1/\varphi(q)$.  Moreover,
\[
 \frac{|G_{s_\nu}(a,q,d)|}{\varphi(q^*_{d})}\leq1.
\]
In the same way,
\eqref{beta-extension} gives the following bound after summation over
$a$, $d$, and $q$.
\begin{align}
 \mathcal E_{\mathrm{ext}}
 &\ll Q^{1/2}\sum_{q\leq Q_1}q^{1/2}
       \sum_{d\mid q}\left(\frac xd+1\right)^{1/2}\notag\\
 &\ll (xQ)^{1/2}Q_1^{3/2}(\log(2Q_1))^2
 \ll (xQ)^{1/2}(\log x)^{C_1}.                 \label{extension-error-sum}
\end{align}
Here we used $Q_1=(\log x)^c$.  We enlarged $C_1$ when needed.

We now find the main coefficient.
The factor $\mu(q)$ restricts $q$ to be squarefree.  Hence
$(d,q^*)=1$ for every $d\mid q$, and therefore
\[
 d_{q^*}=d,\qquad q^*_{d}=q^*=\frac qd.
\]
The classes with $(r,q)=d$
give
\[
 \sum_{d\mid q}G_s(a,q,d)
 =\sum_{r\bmod q}e\left(\frac{asr^2}{q}\right).
\]
Thus the main contribution for a fixed modulus $q$ is
\begin{equation*}\label{fixed-q-principal-contribution}
 \frac{\mu(q)H_\nu(x,k)}{q\varphi(q)}
 \sum_{a\bmod q}^{*}\sum_{r\bmod q}
 e\left(\frac{a(s_\nu r^2+t_\nu k)}q\right).
\end{equation*}
Define the rational complete sum as follows.
\[
 \Sigma_\nu(q,k)=
 \sum_{a\bmod q}^{*}\sum_{r\bmod q}
 e\left(\frac{a(s_\nu r^2+t_\nu k)}q\right).
\]
This sum is multiplicative.  At an odd prime $p$, we have
\[
 \Sigma_+(p,k)=p\left(\frac{-k}{p}\right),\qquad
 \Sigma_-(p,k)=\Sigma_0(p,k)=p\left(\frac{k}{p}\right),
\]
and all three sums vanish at $p=2$.  Hence, for odd squarefree $q$, we
have
\[
 \Sigma_\nu(q,k)=q\left(\frac{\varkappa_\nu k}{q}\right).
\]
We have proved the following formula for each $k$.
\begin{align}
 &\int_{\majorarc}T_{1,z_\nu}(\alpha)T_{2,s_\nu}(\alpha)
 e(t_\nu k\alpha)\,\dif\alpha\notag\\
 &\quad=H_\nu(x,k)
 \sum_{\substack{q\leq Q_1\\2\nmid q}}
 \frac{\mu(q)}{\varphi(q)}
 \left(\frac{\varkappa_\nu k}{q}\right)
 +O(\mathcal E_{\mathrm{count}}+\mathcal E_{\mathrm{ext}})
 \notag\\
 &\quad=H_\nu(x,k)
 \sum_{\substack{q\leq Q_1\\2\nmid q}}
 \frac{\mu(q)}{\varphi(q)}
 \left(\frac{\varkappa_\nu k}{q}\right)
 +O((xQ)^{1/2}(\log x)^{C_1}).                 \label{principal-pointwise}
\end{align}
Let $D_\nu(k)$ be the expression inside the absolute value on the left
of \eqref{principal-L2}.
By \eqref{principal-pointwise} and the definition of
$\mathcal T_{\varkappa_\nu}(k,Q_1)$ in Lemma \ref{lem:tail-bound},
\[
 D_\nu(k)
 =-H_\nu(x,k)\mathcal T_{\varkappa_\nu}(k,Q_1)
 +O(\mathcal E_{\mathrm{count}}+\mathcal E_{\mathrm{ext}}).
\]
Lemma~\ref{lem:tail-bound} and $H_\nu(x,k)\leq x$ now give
\begin{align*}
 \sum_{\substack{k\leq y\\\mu^2(k)=1\\k>1}}|D_\nu(k)|^2
 &\ll x^2\sum_{\substack{k\leq y\\\mu^2(k)=1\\k>1}}
       |\mathcal T_{\varkappa_\nu}(k,Q_1)|^2
       +y\mathcal E_{\mathrm{count}}^2
       +y\mathcal E_{\mathrm{ext}}^2\\
 &\ll \frac{yx^2}{(\log x)^C}
       +yxQ(\log x)^{C_1}.
\end{align*}
We enlarge $C_1$ once more.  It then absorbs the square of the
logarithmic factor in \eqref{extension-error-sum}.  Also, for every
fixed $C$ and all sufficiently large $x$, we have
\[
 y\mathcal E_{\mathrm{count}}^2
 \ll y(\log x)^{2C_1}
 \ll\frac{yx^2}{(\log x)^C}.
\]
This proves \eqref{principal-L2}.
\end{proof}

\subsection{Nonprincipal major arc terms}

\begin{lemma}[Quadratic nonprincipal term]\label{lem:E2}
Let $z$ satisfy \eqref{z-range}.
There is a constant
$C_2=C_2(A,c,K_0)>0$ such that, uniformly for $s\in\{-1,+1\}$,
\begin{align}
 \int_{\majorarc}
 |T_{1,z}(\alpha)E_{2,s}(\alpha)|^2\,\dif\alpha
 &\ll_{A,c,K_0}
 \left(z+\frac{x^4}{Q}\right)(\log x)^{C_2}
\label{weighted-E2-bound}\\
 &\ll_{A,c,K_0}
 \frac{z^2x}{Q^2}(\log x)^{C_2}.             \label{weighted-E2-target}
\end{align}
\end{lemma}

\begin{proof}
The factor $\mu(q)$ in $T_{1,z}$ means that we only need to consider squarefree $q$.
Fix such a $q$ and write
\[
 \alpha=\frac aq+\beta,
 \qquad
 |\beta|\leq\frac1{qQ}.
\]
For every $d\mid q$, put
\[
 r=\frac qd=q^*.
\]
Since $q$ is squarefree, we have
\[
 (d,r)=1,\qquad
 d_{q^*}=d,\qquad
 q^*_{d}=r.
\]
The condition $(a,q)=1$ also gives $(sad,r)=1$.
We also have
\[
 (h,q)=d
 \quad\Longleftrightarrow\quad
 h=du\ \hbox{ and }\ (u,r)=1.
\]
A Dirichlet character modulo $r$ vanishes outside the reduced residue classes.
Thus the inner sum in \eqref{E2-definition} becomes
\[
 F_{d,\chi}(\beta)
 :=
 \sum_{u\leq x/d}
 \overline{\chi}^{\,2}(u)e(s\beta d^2u^2).
\]

Suppose that $\chi^2\ne\chi_0$ modulo $r$.
The character sum bound of P\'olya and Vinogradov is also valid for an
imprimitive nonprincipal character.
It gives
\[
 \max_{U\geq1}
 \left|\sum_{u\leq U}\overline{\chi}^{\,2}(u)\right|
 \ll r^{1/2}\log(2r).
\]
Partial summation then gives
\begin{align}
 |F_{d,\chi}(\beta)|
 &\ll r^{1/2}\log(2r)
 \left(1+|\beta|d^2\left(\frac xd\right)^2\right)\notag\\
 &\ll r^{1/2}\log(2r)(1+|\beta|x^2).
\label{F-character-bound}
\end{align}

Since $r$ is squarefree, the Chinese remainder theorem shows that
$\tau_r(\chi)$ factors as a product of local Gauss sums over the primes
$p\mid r$.
A principal local component has absolute value $1$, while a
nonprincipal local component has absolute value $\sqrt p$.
Hence
\begin{equation}\label{squarefree-Gauss-bound}
 |\tau_r(\chi)|\leq r^{1/2}
\end{equation}
for every character $\chi$ modulo $r$.

Using \eqref{F-character-bound} and
\eqref{squarefree-Gauss-bound} in \eqref{E2-definition}, we obtain
\begin{align}
 \left|E_{2,s}\left(\frac aq+\beta\right)\right|
 &\leq
 \sum_{d\mid q}\frac1{\varphi(r)}
 \sum_{\substack{\chi\bmod r\\\chi^2\ne\chi_0}}
 |\tau_r(\chi)|\,|F_{d,\chi}(\beta)|\notag\\
 &\ll
 \sum_{d\mid q}r\log(2r)(1+|\beta|x^2)\notag\\
 &\ll
 \sigma(q)\log(2q)(1+|\beta|x^2).
    \label{E2-pointwise}
\end{align}
Here $\sigma(q)=\sum_{d\mid q}d$.
This estimate is uniform in $a$ and in the sign $s$.

On the same arc,
\begin{equation}\label{T1-geometric-bound}
 \left|T_{1,z}\left(\frac aq+\beta\right)\right|
 \ll
 \frac1{\varphi(q)}
 \min\left\{z,\frac1{|\beta|}\right\}.
\end{equation}
Put $W(q)=\sigma(q)\log(2q)$.
The major arcs are disjoint.
We first sum over the reduced residue classes $a$ and then use
\eqref{E2-pointwise} and \eqref{T1-geometric-bound}.
This gives
\begin{equation}
 \int_{\majorarc}
 |T_{1,z}(\alpha)E_{2,s}(\alpha)|^2\,\dif\alpha
\ll
 \sum_{\substack{q\leq Q_1\\\mu^2(q)=1}}
 \frac{W(q)^2}{\varphi(q)}
 \int_0^{1/(qQ)}
 \min\left\{z,\frac1\beta\right\}^2
 (1+\beta x^2)^2\,\dif\beta.
\label{weighted-E2-sum}
\end{equation}

For sufficiently large $x$, the inequalities in \eqref{z-range} give
$qQ<z$ for every $q\leq Q_1$.
Splitting the last integral at $\beta=1/z$, we find
\begin{align}
 &\int_0^{1/(qQ)}
 \min\left\{z,\frac1\beta\right\}^2
 (1+\beta x^2)^2\,\dif\beta\notag\\
 &\quad\ll
 z\left(1+\frac{x^2}{z}\right)^2
 +\int_{1/z}^{1/(qQ)}
\left(\beta^{-2}+2x^2\beta^{-1}+x^4\right)\,\dif\beta\notag\\
 &\quad\ll
 z(\log x)^{2A+2}+\frac{x^4}{qQ}.
\label{beta-weighted-bound}
\end{align}
Here we used $x^2/z\ll(\log x)^A$.

Since $q\leq Q_1=(\log x)^c$, elementary divisor estimates give
\[
 \sum_{q\leq Q_1}\frac{W(q)^2}{\varphi(q)}
 \ll(\log x)^{C_2},
 \qquad
 \sum_{q\leq Q_1}\frac{W(q)^2}{q\varphi(q)}
 \ll(\log x)^{C_2}
\]
after enlarging $C_2$.
Equations \eqref{weighted-E2-sum} and \eqref{beta-weighted-bound} prove \eqref{weighted-E2-bound}.

Finally, $Q=x^{1-\delta}$ and $z\geq x^2(\log x)^{-A}$ give both
$z\ll z^2x/Q^2$ and $x^4/Q\ll z^2x/Q^2$.  This proves
\eqref{weighted-E2-target}.
\end{proof}

We next estimate $E_{1,z}$ using the short-interval estimates in
Lemma~\ref{lem:short-interval-input}.

\begin{lemma}[Prime nonprincipal term]\label{lem:E1}
Let $c>0$ be fixed. For $z\in\{z_+,z_-,z_0\}$ and every $C>0$, one has
\begin{equation}\label{E1-bound}
 \int_{\majorarc}|E_{1,z}(\alpha)|^2\,\dif\alpha
 \ll_{C,A,K_0}\frac z{(\log x)^C}.
\end{equation}
\end{lemma}

\begin{proof}
Recall that $0<\delta<\frac1{10}$ in \eqref{circle-parameters}.
Choose $0<\eta_0<\delta$ small enough that
$2(\frac15+\eta_0)<1-\delta$, and put
\[
 R=(2Q)^{1/(1-\eta_0)}.
\]
For $\alpha=a/q+\beta\in\majorarc$, define
\begin{align*}
 E_{1,z}^{\leq R}(\alpha)
  &=\frac1{\varphi(q)}
    \sum_{\chi\bmod q}\tau_q(\chi)\overline\chi(a)
    \sum_{m\leq R}A_\chi(m)e(\beta m),\\
 E_{1,z}^{>R}(\alpha)
  &=E_{1,z}(\alpha)-E_{1,z}^{\leq R}(\alpha).
\end{align*}
For sufficiently large $x$, we have $R<z$.
We split the second expression into $O(\log x)$ dyadic blocks of the form
$N<m\leq N'\leq2N$.

We first record the character orthogonality used for each block.
Put
\[
 F_{\chi,N}(\beta)
 =\sum_{N<m\leq N'}A_\chi(m)e(\beta m),
\]
where $A_\chi$ is defined in \eqref{def Achi}.
Orthogonality on the reduced residue classes and
$|\tau_q(\chi)|\leq q^{1/2}$ give
\begin{equation}\label{E1-character-orthogonality}
 \sum_{a\bmod q}^{*}
 \left|
  \frac1{\varphi(q)}
  \sum_{\chi\bmod q}
  \tau_q(\chi)\overline\chi(a)F_{\chi,N}(\beta)
 \right|^2
 \leq
 \frac q{\varphi(q)}
 \sum_{\chi\bmod q}|F_{\chi,N}(\beta)|^2.
\end{equation}

Conjugation permutes the characters modulo $q$.
Therefore, the sum of the short interval second moments built from
$A_\chi$ is exactly the quantity $J_N(q,\Delta)$ in \eqref{J-definition}.
If $R\leq N\leq z$, then, for all sufficiently large $x$,
\begin{equation*}\label{Wolke-range}
 N^{1/5+\eta_0}\leq Q/2\leq N^{1-\eta_0},
\end{equation*}
and $q\leq Q_1\leq(\log N)^{c+1}$.
The first inequality follows from $N\leq z\ll_{K_0}x^2$ and
$2(1/5+\eta_0)<1-\delta$.
The second follows from $N\geq R$ since $R^{1-\eta_0}=2Q$.
We also have
\begin{equation*}\label{Gallagher-parameter-check}
 \frac{N}{qQ}\geq\frac{R}{Q_1Q}
 \gg\frac{Q^{\eta_0/(1-\eta_0)}}{(\log x)^c}\longrightarrow\infty.
\end{equation*}
Thus $2<qQ<N/2$ as required in
\eqref{Gallagher-exact}.
Lemma~\ref{lem:short-interval-input}, with
$b=c+1$ and $\varepsilon=\eta_0$, gives for every $L>0$
\begin{equation}\label{Wolke-Mikawa-bound}
 J_N(q,Q/2)
 \ll_{L,c,\eta_0}(qQ)^2N(\log N)^{-L}.
\end{equation}

Put $H=qQ$.
The inner sums in \eqref{Gallagher-exact} have length $H/2$.
The formula gives
\begin{equation}\label{Gallagher-truncated}
 I_{\chi,N}
 :=\int_{|\beta|\leq H^{-1}}
 |F_{\chi,N}(\beta)|^2\,\dif\beta
 \ll H^{-2}\int_{N-H/2}^{N'}
 \left|\sum_{\max\{t,N\}<m\leq
 \min\{t+H/2,N'\}}A_\chi(m)\right|^2\,\dif t.
\end{equation}
On the interior range $N\leq t\leq N'-H/2$, the sum in
\eqref{Gallagher-truncated} is the full sum over $t<m\leq t+H/2$.
The two end ranges have total length $O(H)$.
On these ranges, the sum has $O(H)$ terms and
$|A_\chi(m)|\ll\log x$.
Thus their total contribution before multiplication by $H^{-2}$ is
$O(H^3(\log x)^2)$.
It follows that
\begin{align*}
 I_{\chi,N}
 &\ll \frac1{(qQ)^2}\int_N^{2N}
 \left|\sum_{t<m\leq t+qQ/2}A_\chi(m)\right|^2\,\dif t
 +qQ(\log x)^2.                              \label{Gallagher-block}
\end{align*}
Summing over the characters and using again that conjugation permutes
them, we obtain
\begin{equation}\label{Gallagher-character-sum}
 \sum_{\chi\bmod q}I_{\chi,N}
 \ll \frac{J_N(q,Q/2)}{(qQ)^2}
      +\varphi(q)qQ(\log x)^2.
\end{equation}

Let $D_0\ll\log x$ be the number of dyadic blocks.
Cauchy's inequality over these blocks, followed by
\eqref{E1-character-orthogonality} and
\eqref{Gallagher-character-sum}, gives
\begin{align}
 \int_{\majorarc}|E_{1,z}^{>R}(\alpha)|^2\,\dif\alpha
 &\ll D_0
 \sum_{\substack{N\ \mathrm{dyadic}\\R\leq N\leq z}}
 \sum_{q\leq Q_1}\frac q{\varphi(q)}
 \frac{J_N(q,Q/2)}{(qQ)^2}
 +D_0^2Q_1^3Q(\log x)^2.
 \label{E1-long-blocks}
\end{align}
For one block and one $q$, the endpoint contribution is at most
$q^2Q(\log x)^2$.
We also have $\sum_{q\leq Q_1}q^2\ll Q_1^3$.
We now use \eqref{Wolke-Mikawa-bound} together with
\[
 \sum_{q\leq Q_1}\frac q{\varphi(q)}
 \ll Q_1\log(2Q_1),
 \qquad
 \sum_{\substack{N\ \mathrm{dyadic}\\R\leq N\leq z}}N\ll z,
\]
and obtain that the first term in \eqref{E1-long-blocks} is
$O_{C,A,K_0}(z(\log x)^{-C})$ if we take, for example,
$L>C+c+5$.
The second term is $x^{1-\delta}(\log x)^{O(c)}$.
It has the same bound because $z\geq x^2(\log x)^{-A}$.
This is the full calculation used in
\cite[(6.6) to (6.8)]{BaierZhao2007}.

For the short part, the character decomposition
\eqref{S1-decomposition} with cutoff $R$ gives the identity
\begin{equation*}\label{E1-short-identity}
 E_{1,z}^{\leq R}\left(\frac aq+\beta\right)
 =S_{1,R}\left(\frac aq+\beta\right)
 -T_{1,R}\left(\frac aq+\beta\right)
  +O((\log R)^2).
\end{equation*}
Here the error is uniform for $q\leq Q_1$.
The terms excluded by $(m,q)=1$ have total weight at most
\[
 \sum_{p\mid q}\sum_{\substack{j\geq1\\p^j\leq R}}\log p
 \leq\omega(q)\log R\ll(\log R)^2.
\]
Parseval's identity, the bound $|T_{1,R}|\leq R$, and
$\operatorname{meas}(\majorarc)\ll Q_1/Q$ give
\begin{align*}
 \int_{\majorarc}|E_{1,z}^{\leq R}(\alpha)|^2\,\dif\alpha
 &\ll \sum_{m\leq R}\Lambda(m)^2
      +\frac{Q_1R^2}{Q}
      +\frac{Q_1}{Q}(\log R)^4\notag\\
 &\ll R\log(2R)+\frac{Q_1R^2}{Q}.
 \label{E1-short-blocks}
\end{align*}

Finally,
\[
 \frac{R^2}{Q}
 =2^{2/(1-\eta_0)}Q^{(1+\eta_0)/(1-\eta_0)}
 \ll x^{2-\eta_1}
\]
for some $\eta_1>0$.
Also $R\ll x^{1-\eta_2}$ for some $\eta_2>0$.
The endpoint term in \eqref{E1-long-blocks} is
$O(x^{1-\delta}(\log x)^{O(c)})$.
Since $z\geq x^2(\log x)^{-A}$, all these power saving terms are
\[
 O_{C,A,K_0}\bigl(z(\log x)^{-C}\bigr).
\]
This proves \eqref{E1-bound}.
\end{proof}

Combining \eqref{S1-decomposition} and \eqref{S2-decomposition}
and regrouping the terms, we obtain
\begin{equation}\label{recombined-major-decomposition}
 S_{1,z}S_{2,s}
 =T_{1,z}T_{2,s}+T_{1,z}E_{2,s}+E_{1,z}S_{2,s}
 +R_{q,z}S_{2,s}.
\end{equation}
Bessel's inequality and the preceding estimates give the following
bounds for the remaining major arc contributions.

\begin{lemma}[Remaining major arc terms]\label{lem:mixed-major}
For every $C>0$ and $z=z_\nu$, one has
\begin{align}
 \sum_{k\leq y}
 \left|
 \int_{\majorarc}T_{1,z}E_{2,s_\nu}
 e(t_\nu k\alpha)\,\dif\alpha
 \right|^2
 &\ll
 \frac{z^2x}{Q^2}(\log x)^{C_2},
\label{T1E2-bound}\\
 \sum_{k\leq y}
 \left|
 \int_{\majorarc}E_{1,z}S_{2,s_\nu}
 e(t_\nu k\alpha)\,\dif\alpha
 \right|^2
 &\ll_{C,A,K_0}
 \frac{x^2z}{(\log x)^C}.                     \label{E1S2-bound}
\end{align}
The remainder $R_{q,z}$ in \eqref{S1-decomposition} contributes $O(\frac{yx^2Q_1^2}{Q^2}(\log x)^4)$
to the total second moment.
\end{lemma}

\begin{proof}
Extend the integrands by zero to $\mathcal I$.
Bessel's inequality and Lemma~\ref{lem:E2} give \eqref{T1E2-bound}.
The same argument, with $|S_{2,s_\nu}|\leq x$ and
Lemma~\ref{lem:E1}, gives \eqref{E1S2-bound}.
Finally, $\operatorname{meas}(\majorarc)\ll Q_1/Q$.
Thus the contribution of $R_{q,z}S_{2,s_\nu}$ is
$O(xQ_1(\log x)^2/Q)$ for each $k$.
Squaring and summing proves the last assertion.
\end{proof}

\subsection{Minor arcs and completion of the mean square estimates}

\begin{lemma}[Minor arc contribution]\label{lem:minor-arcs}
Uniformly in $\nu\in\{+,-,0\}$,
\begin{equation}
\sum_{k\leq y}\left|\int_{\minorarc}
 S_{1,z_\nu}(\alpha)S_{2,s_\nu}(\alpha)e(t_\nu k\alpha)
 \,\dif\alpha\right|^2
\ll z_\nu\log(2z_\nu)(\log x)^2
 \left(\frac{x^2}{Q_1}+Qx\right).              \label{minor-error-general}
\end{equation}
\end{lemma}

\begin{proof}
Extend the integrand by zero to $\mathcal I$.
Bessel's inequality bounds the left side by
\[
 \sup_{\alpha\in\minorarc}|S_{2,s_\nu}(\alpha)|^2
 \int_0^1|S_{1,z_\nu}(\alpha)|^2\,\dif\alpha.
\]
Weyl's inequality and Dirichlet approximation bound the first factor
by $(\log x)^2(x^2/Q_1+Qx)$.
This bound does not depend on the sign $s_\nu$.
Parseval's identity bounds the second factor by $z_\nu\log(2z_\nu)$.
This proves the lemma.
\end{proof}

\begin{proof}[Proof of Theorems~\ref{thm:positive-main},
\ref{thm:negative-main}, and~\ref{thm:reverse-main}]
We combine Lemmas~\ref{lem:principal-arcs}, \ref{lem:mixed-major}, and
\ref{lem:minor-arcs}.
The grouping in \eqref{recombined-major-decomposition} contains every major arc term.
Write
\[
 \mathcal E=\frac{yx^2}{(\log x)^B}.
\]
First take a saving larger than $B+5$ in
Lemma~\ref{lem:tail-bound} and a saving larger than $A+B+5$ in
Lemma~\ref{lem:E1}.
Next choose $c$ large enough for the first choice and also so that
$c>A+B+10$.
After fixing $c$, the exponents $C_1$ and $C_2$ in the previous lemmas are fixed.

Since
\[
 \mathcal E
 \geq\frac{x^4}{(\log x)^{A+B}}
\]
and $z_\nu\ll_{K_0}x^2$, the right side of
\eqref{T1E2-bound} is
\[
 O\left(x^{3+2\delta}(\log x)^{C_2}\right)
 =o\left(\frac{x^4}{(\log x)^{A+B}}\right),
\]
because $2\delta<1$.
For \eqref{E1S2-bound}, the ratio to $\mathcal E$ is at most
\[
 \frac{z_\nu}{y}(\log x)^{B-C}
 \ll_{A,K_0}(\log x)^{A+B-C},
\]
which tends to zero with the saving chosen above.

The second term in \eqref{principal-L2} and the $Qx$ term in
\eqref{minor-error-general}, after division by $\mathcal E$, are
\[
 O\left(x^{-\delta}(\log x)^{A+B+O_c(1)}\right)=o(1).
\]
The $x^2/Q_1$ term in \eqref{minor-error-general} divided by
$\mathcal E$ is
\[
 O\left((\log x)^{A+B+3-c}\right)=o(1).
\]
Finally, the second moment contribution of the remainder $R_{q,z}$ in \eqref{S1-decomposition} divided by
$\mathcal E$ is
\[
 O\left(\frac{Q_1^2(\log x)^{B+4}}{Q^2}\right)
 =O\left(x^{-2+2\delta}(\log x)^{2c+B+4}\right)=o(1).
\]
The principal arc lemmas do not include $k=1$.
This value occurs only in the positive-shift theorem.
The trivial estimate makes its squared error $O(x^2(\log x)^2)$.
This is $o(x^4(\log x)^{-A-B})$.
This proves all three estimates.
The constants are uniform in the stated ranges.
\end{proof}
\section{From prime values to prime index ratios}\label{sec:counting}

We first record the pointwise interval consequences of the three mean
square theorems.  The first lemma fixes the prime $q$ and treats
$h^2+q$ and $h^2-q$.  The second lemma fixes the larger prime $P$ and
treats $P-h^2$.

\begin{lemma}[Box estimate with a fixed shift]
\label{lem:good-small}
Fix $0<b_1<b_2$ and $0<c_1<c_2$.  Given $C>0$, all but
$O_C(X(\log X)^{-C})$ primes $q\in(b_1X,b_2X]$ have a prime of the form
\begin{equation}\label{plus-prime-in-box}
 p=h^2+q,
 \qquad c_1\sqrt X<h\leq c_2\sqrt X.
\end{equation}
If $b_2<c_1^2$, the same result holds with
\begin{equation}\label{minus-prime-in-box}
 p=h^2-q,
 \qquad c_1\sqrt X<h\leq c_2\sqrt X.
\end{equation}
In both cases, every nonexceptional $q$ has
\begin{equation}\label{many-small-endpoints}
 \gg_{b_1,b_2,c_1,c_2}\frac{\sqrt X}{(\log X)^2}
\end{equation}
values of $h$ for which the represented value is prime.  The implied
constants may depend on the fixed box.
\end{lemma}

\begin{proof}
For~\eqref{plus-prime-in-box}, choose
\begin{equation}\label{small-box-K0}
 K_0>\max\left\{1,\frac{b_2}{c_1^2}\right\}.
\end{equation}
Set $A=1$ and choose a fixed $D>2$.  We apply
Corollary~\ref{cor:pointwise} with the given exponent $C$.  The squarefree condition in Corollary~\ref{cor:pointwise}
is automatic when the shift is prime.  Apply
\eqref{positive-absolute-corollary} at $c_1\sqrt X$ and
$c_2\sqrt X$, with $y=b_2X$ and this fixed $K_0$.  All required
inequalities hold for sufficiently large $X$.  The union of the two
exceptional sets has size $O_C(X(\log X)^{-C})$.  Subtraction gives,
for every remaining prime $q$,
\begin{equation*}\label{positive-pointwise-interval}
\sum_{c_1\sqrt X<h\leq c_2\sqrt X}
  \Lambda(h^2+q)
=\mathfrak S_+(q)(c_2-c_1)\sqrt X
  +O_D\left(\frac{\sqrt X}{(\log X)^D}\right).
\end{equation*}

For~\eqref{minus-prime-in-box}, apply
\eqref{negative-absolute-corollary} at the same two endpoints.  The
condition $b_2<c_1^2$ ensures that $y=b_2X$ lies in the range of
Theorem~\ref{thm:negative-main}.  It also gives
$(c_j\sqrt X-\sqrt q)_+=c_j\sqrt X-\sqrt q$.  The two terms
$-\sqrt q$ therefore cancel, and we obtain
\begin{equation*}\label{negative-pointwise-interval}
 \sum_{c_1\sqrt X<h\leq c_2\sqrt X}
  \Lambda(h^2-q)
=\mathfrak S_-(q)(c_2-c_1)\sqrt X
  +O_D\left(\frac{\sqrt X}{(\log X)^D}\right).
\end{equation*}
By~\eqref{series-lower}, the main term in both formulas is
$\gg\sqrt X/\log X$.

It remains to remove the contribution of proper prime powers.  Suppose
$h^2\pm q=r^a$ with $a\geq2$.
If $a=2d$, then
\[
 (r^d-h)(r^d+h)=q\quad(h^2+q=r^{2d}),
\]
or
\[
 (h-r^d)(h+r^d)=q\quad(h^2-q=r^{2d}).
\]
Since $q$ is prime, each even exponent gives at most one solution of
each type.  There are $O(\log X)$ possible even exponents, so their
total von Mangoldt weight is $O((\log X)^2)$.  For odd exponents, each
pair $(a,r)$ gives at most one $h$.  Hence
\[
 \sum_{\substack{3\leq a\ll\log X\\a\ \mathrm{odd}}}
 \ \,\sum_{r\ll X^{1/a}}\log r
 \ll\sum_{\substack{3\leq a\ll\log X\\a\ \mathrm{odd}}}
 X^{1/a}\log X
 \ll X^{1/3}\log X.
\]
Thus all proper prime powers have total von Mangoldt weight
\begin{equation}\label{prime-power-bound}
 O\bigl(X^{1/3}\log X+(\log X)^2\bigr)
 =O(X^{1/3}\log X).
\end{equation}
For our choice of $D$, the two main terms dominate their errors and
the bound in \eqref{prime-power-bound}.  In the positive case, all
represented values lie between fixed positive multiples of $X$.  The
same is true in the negative case since $c_1^2-b_2>0$.  The prime
values therefore have size $\asymp X$ and total weight
$\gg\sqrt X/\log X$.  Each prime has weight $O(\log X)$, which proves
\eqref{many-small-endpoints} and the existence of a represented
prime.  The map $h\mapsto h^2\pm q$ is strictly increasing for $h>0$.
Thus different values of $h$ give different primes and different
prime indices.
\end{proof}

\begin{lemma}[Box estimate with a fixed larger endpoint]
\label{lem:good-large}
Fix constants
\begin{equation}\label{large-box-constants}
 0<a_1<a_2,\qquad 0<c_1<c_2,\qquad c_2^2<a_1.
\end{equation}
Given $C>0$, all but $O_C(X(\log X)^{-C})$ primes
$P\in(a_1X,a_2X]$ have a smaller prime $q$ of the form
\begin{equation*}\label{reverse-prime-in-box}
 q=P-h^2,
 \qquad c_1\sqrt X<h\leq c_2\sqrt X.
\end{equation*}
For every such nonexceptional $P$, the number of $h$ for which
$P-h^2$ is prime is
\begin{equation}\label{many-large-endpoints}
 \gg_{a_1,a_2,c_1,c_2}\frac{\sqrt X}{(\log X)^2}.
\end{equation}
The implied constants may depend on the fixed box.
\end{lemma}

\begin{proof}
Choose
\begin{equation*}\label{large-box-K0}
 K_0>\max\left\{1,\frac{a_2}{c_1^2}\right\}.
\end{equation*}
Set $A=1$ and choose a fixed $D>2$. For each prime $P\in(a_1X,a_2X]$, the shift $k=P$ satisfies
the squarefree condition in Corollary~\ref{cor:pointwise}. Apply
\eqref{reverse-absolute-corollary} at $c_1\sqrt X$ and
$c_2\sqrt X$, with the given exponent $C$, $y=a_2X$ and this fixed $K_0$.  For every prime
outside the two exceptional sets, \eqref{large-box-constants} gives
$H_0(c_j\sqrt X,P)=c_j\sqrt X$.  Subtraction gives
\begin{equation}\label{reverse-pointwise-interval}
 \sum_{c_1\sqrt X<h\leq c_2\sqrt X}
  \Lambda(P-h^2)
=\mathfrak S_-(P)(c_2-c_1)\sqrt X
  +O_D\left(\frac{\sqrt X}{(\log X)^D}\right).
\end{equation}
By~\eqref{series-lower}, the main term is $\gg\sqrt X/\log X$.

Suppose now that $P-h^2=r^a$ is a proper prime power.
If $a=2d$, then
\[
 P=h^2+(r^d)^2.
\]
A prime has at most eight ordered signed representations as a sum of
two squares.  Since the second coordinate is a prime power, it
determines $r$ and $d$ uniquely.  Thus the even exponents give
weight $O(\log X)$, which is also $O((\log X)^2)$.  For odd
$a\geq3$, the proof of \eqref{prime-power-bound} gives weight
$O(X^{1/3}\log X)$.  Since $D>2$, the prime values in
\eqref{reverse-pointwise-interval} still have total weight
$\gg\sqrt X/\log X$.
Since
\[
 (a_1-c_2^2)X<P-h^2<(a_2-c_1^2)X,
\]
all these values have size $\asymp X$.  Since each prime term has weight $O(\log X)$,
\eqref{many-large-endpoints} follows.  The map $h\mapsto P-h^2$ is
strictly decreasing for $h>0$.  Thus different values of $h$ give
different primes and different indices.
\end{proof}

We now pass from one prime block to all primes.  Choose $X_0$ large
enough that the box estimate holds for every $X\geq X_0$.  Define
$X_j$ by $u_2X_j=u_1X_{j+1}$, where $0<u_1<u_2$ are the fixed
endpoints of the prime interval.  The blocks
\[
 \mathcal B_j=(u_1X_j,u_2X_j]
\]
are disjoint and cover $(u_1X_0,\infty)$.  If $T$ lies in
$\mathcal B_J$, the number of exceptional primes below $T$ in this
block is at most the number in the full block $\mathcal B_J$.  The
primes below $u_1X_0$ give only $O(1)$ terms.  The relevant box lemma
gives $O_{C,u_1,u_2}(X(\log X)^{-C})$ exceptional primes in every full
block.
Summing over the block indices $j$, we obtain
\begin{equation}\label{exception-sum}
 \#\{p\leq T:p\text{ is an exceptional prime}\}
 \ll_C 1+
 \sum_{\substack{j\geq0\\u_1X_j\leq T}}
 \frac{X_j}{(\log X_j)^C}
 \ll_C \frac{T}{(\log T)^C}+T^{1/2},
\end{equation}
where the implied constants may also depend on the fixed parameters
$u_1,u_2,c_1,c_2$.
For the last estimate, split the sum into the ranges
$X_j\leq T^{1/2}$ and $X_j>T^{1/2}$.
The first contributes $O(T^{1/2})$ by geometric summation.
In the second, $\log X_j\asymp\log T$, and the corresponding
$X_j$ have total sum $O(T)$.
Thus the number of exceptional primes up to $T$ is
$o(\pi(T))$ for $C>2$, by the prime number theorem.

We also use the multiplicity given by the two box lemmas.  Their
proofs show that different values of $h$ give different partner
primes.  After division by $X$, all endpoints stay in fixed compact
intervals inside $(0,\infty)$.  If the fixed endpoint is
$p_j\asymp X$, the prime number theorem gives, uniformly in the box,
\begin{equation}\label{index-multiplicity-transfer}
 X\asymp p_j\asymp j\log(2j),\qquad
 \log X\asymp\log(2j),\qquad
 \frac{\sqrt X}{(\log X)^2}
 \asymp\frac{\sqrt j}{(\log(2j))^{3/2}}.
\end{equation}
Fix $L>0$ and choose $C>L+1$ in the box estimates.
Taking $T=p_N\sim N\log N$, we obtain from
\eqref{exception-sum}
\[
 \#\{j\leq N:p_j\text{ is an exceptional prime}\}
 \ll_{I,L}
 \frac{N}{(\log N)^{C-1}}+\sqrt{N\log N}
 \ll_{I,L}\frac{N}{(\log N)^L}.
\]
Every sufficiently large nonexceptional prime $p_j$ lies in
one of the fixed boxes and has the partner lower bound
given by \eqref{index-multiplicity-transfer}.
The finitely many remaining indices are absorbed into
the error term.

The constant in the partner lower bound can be chosen
independently of $L$.  Indeed, in each box lemma we fix
$D>2$ and obtain the lower bound by retaining a fixed
positive proportion of the singular series main term.
Increasing $C$ changes the exceptional set estimate, but does not change this proportion.
Consequently, the partner lower bound fails for at most
\[
 O_{I,L}\left(\frac{N}{(\log N)^L}\right)
\]
indices $j\leq N$.

\begin{proof}[Proof of Theorem~\ref{thm:main}]
Let $I$ be a nonempty open interval.
We first consider the sum condition.
Choose $R\in I$ and a compact interval $J\subset I$ such that
$R\in\operatorname{int}(J)$.
By continuity, we may choose
\begin{equation*}\label{sum-box-constants}
 0<b_1<b_2<c_1^2<c_2^2
\end{equation*}
with $b_1$ and $b_2$ close to $1$, and with $c_1^2$ and $c_2^2$
close to $R+1$, so that
\begin{equation}\label{sum-ratio-box}
 \left[
  \frac{c_1^2-b_2}{b_2},
  \frac{c_2^2-b_1}{b_1}
 \right]\subset J.
\end{equation}
Apply Lemma~\ref{lem:good-small} to $h^2-q$.
For all but $O_C(X(\log X)^{-C})$ primes $q\in(b_1X,b_2X]$,
there are $\gg_I\sqrt X/(\log X)^2$ distinct primes $p=h^2-q$
with $h$ in the chosen box.
Put $m=\pi(p)$ and $n=\pi(q)$.
Then $p_m+p_n=h^2$, and the size ratio $p/q$ lies in the compact
interval in \eqref{sum-ratio-box}, which is contained in
$J\subset I$.
Moreover, both $p/X$ and $q/X$ remain in fixed compact subsets of
$(0,\infty)$.
The uniform prime number theorem in \eqref{index-transfer} therefore
gives
\[
 \frac mn=\frac{\pi(p)}{\pi(q)}=\frac pq+o(1)\in I
\]
for all sufficiently large $X$.
Combining \eqref{exception-sum} and
\eqref{index-multiplicity-transfer}, and taking $T=p_N$, proves
\eqref{sum-many-D}.
Indeed, choosing $C>L+1$ gives
$O_{I,L}(N(\log N)^{-L})$ exceptional indices, as shown above.

Apply this result to $I^{-1}=\{t^{-1}\mid t\in I\}$ and exchange the
two indices.
This proves \eqref{sum-many-N} because
\[
 r_{\Sigma,D}(j,I^{-1})=r_{\Sigma,N}(j,I).
\]

We next consider the difference condition.
If $1\in I$, choose a nonempty open interval
$I_+\subset I\cap(1,\infty)$.
Since each counting function is nondecreasing under inclusion
of intervals, the result for $I_+$ implies the result for $I$.
It therefore suffices to consider
$I\subset(0,1)$ or $I\subset(1,\infty)$.

Suppose now that $I\subset(1,\infty)$.
Choose $R\in I$ and a compact interval $J\subset I$ such that
$R\in\operatorname{int}(J)$.

For the denominator estimate, choose $0<b_1<b_2$ close to $1$ and
$0<c_1<c_2$ such that $c_j^2$ is close to $R-1$.
We choose these intervals small enough that
\begin{equation}\label{difference-small-box}
 \left[
  1+\frac{c_1^2}{b_2},
  1+\frac{c_2^2}{b_1}
 \right]\subset J.
\end{equation}
No relation between $b_2$ and $c_1^2$ is needed.
When $R$ is close to $1$, we use a larger fixed value of $K_0$ in
\eqref{small-box-K0}.
Apply Lemma~\ref{lem:good-small} to $h^2+q$.
For every good prime $q=p_n$, there are
$\gg_I\sqrt X/(\log X)^2$ distinct primes $P=q+h^2=p_m$.
The size ratio $P/q$ lies in the compact interval in
\eqref{difference-small-box}.
Uniformly over the box, \eqref{index-transfer} gives
$m/n=P/q+o(1)\in I$.
Since $|p_m-p_n|=h^2$, equations \eqref{exception-sum} and
\eqref{index-multiplicity-transfer} prove
\eqref{difference-many-D}.

For the numerator estimate, choose $0<a_1<a_2$ close to $R$ and
$0<c_1<c_2$ such that $c_j^2$ is close to $R-1$.
We choose these intervals small enough that
\begin{equation}\label{difference-large-box}
 c_2^2<a_1,\qquad
 \left[
  \frac{a_2}{a_2-c_1^2},
  \frac{a_1}{a_1-c_2^2}
 \right]\subset J.
\end{equation}
These choices are possible because both ratios in the display are
equal to $R$ at the central values.
Lemma~\ref{lem:good-large} gives
$\gg_I\sqrt X/(\log X)^2$ distinct primes $q=P-h^2=p_n$ for every
good prime $P=p_m$.
The size ratio $P/q$ lies in the compact interval in
\eqref{difference-large-box}.
Thus \eqref{index-transfer} gives $m/n=P/q+o(1)\in I$ for all
sufficiently large $X$.
Equations \eqref{difference-large-box}, \eqref{exception-sum}, and
\eqref{index-multiplicity-transfer} prove
\eqref{difference-many-N}.

Finally, suppose that $I\subset(0,1)$.
Then $I^{-1}\subset(1,\infty)$.
After exchanging $m$ and $n$, the numerator estimate for $I^{-1}$
gives the denominator estimate for $I$, while the denominator
estimate for $I^{-1}$ gives the numerator estimate for $I$.
Indeed,
\[
 r_{\Delta,D}(j,I)=r_{\Delta,N}(j,I^{-1}),
 \qquad
 r_{\Delta,N}(j,I)=r_{\Delta,D}(j,I^{-1}).
\]
Thus both difference estimates hold for every nonempty open interval
$I\subset\mathbb R_{>0}$.

Taking $c_I$ to be the smallest of the four positive constants above
proves \eqref{sum-many-D}, \eqref{sum-many-N},
\eqref{difference-many-D}, and \eqref{difference-many-N}.
This completes the proof.
\end{proof}

\section*{Acknowledgements}
S.H. would like to thank Bingrong Huang and Chung Pang Mok for their constant encouragement.
S.X. would like to thank Jianya Liu and Ye Tian for their constant encouragement.
S.H. was partially supported by NSFC (No. 12601008), and the China Postdoctoral Science Foundation (No. 2026M793348).
S.X. was partially supported by NSFC (No. 12631001).

This project was started at the 2026 Qiuzhen International Mathematics Conference held at Tsinghua University.

\section*{Author contributions}

Both authors contributed equally to this work.
The authors are listed in alphabetical order and serve as co-corresponding authors.

\section*{Conflict of interest statement}

The authors declare no conflict of interest.

\end{document}